\documentclass[a4, 11pt]{amsart}
\usepackage[utf8]{inputenc}
\usepackage{amsfonts,mathrsfs,amsmath,amscd,amssymb}
\usepackage[pdftex]{graphicx}
\usepackage{hyperref}
\usepackage{color}
\newtheorem{theorem}{Theorem}[section]
\usepackage[utf8]{inputenc}
\usepackage{enumitem}
\usepackage{mathtools}
\usepackage[english]{babel}
\usepackage[all,cmtip]{xy}
\usepackage{soul} 

\newtheorem{proposition}[theorem]{\sf Proposition}
\newtheorem{lemma}[theorem]{\sf Lemma}
\newtheorem{definition}[theorem]{\sf Definition}
\newtheorem{corollary}[theorem]{\sf Corollary}
\newtheorem{remark}[theorem]{\sf Remark}

\def\C{\mathbb C}
\def \Z{\mathbb Z}

\def \P{\mathbb P}
\def \L{\mathbb L}

\def \ra{\rightarrow}

\def \ra{\rightarrow}

\let\olddefinition\definition
\renewcommand{\definition}{\olddefinition\normalfont}

\let\oldremark\remark
\renewcommand{\remark}{\oldremark\normalfont}

\title{An example for Kuznetsov-Shinder conjecture}
\author{Tanya Kaushal Srivastava}
\address{Indian Institute of Technology Gandhinagar, Gandhinagar-382355, India.}
\email{ks.tanya@iitgn.ac.in}
\date{\today}

\begin{document}

\maketitle
\begin{abstract}
    We give an example for the Kuznetsov-Shinder conjecture, with infinitely many non-isomorphic but L-equivalent varieties.
\end{abstract}

\section{The Grothendieck ring of varieties}

Let $k$ be a field of characteristic zero. One define the Grothendieck ring of varieties over $k$ as the ring generated by isomorphism classes $[X]$ of algebraic varieties $X$ over $k$ with relations $[X]=[Z]+[U]$ for every closed subvariety $Z \subset X$ having open complement $U \subset X$. The product structure of this ring is induced by the products of varieties. The ring is denoted by $K_0(Var/k)$. The ring has a unit element given by the equivalence class of a point $[Spec (k)]$. This ring was introduced by Grothendieck in his correspondence with Serre \cite{Grothendieck-serre}. Results of Bittner \cite{Bittner} and Looijenga \cite{Looi} give an alternate description of this ring, showing that the ring is generated by isomorphism classes of smooth projective varieties subject to the relation: $[X] + [Z] = [Y] + [E]$, whenever $Y$ is smooth projective, $Z \subset Y$ is a smooth subvariety, $X$ is a blow-up of $Y$ along $Z$, and $E \subset X$ is the exceptional divisor. The following lemmas are very useful and  easy consequences of the definition of Grothendieck ring of varieties. 

\begin{lemma} \cite[Lemma 2.1]{KS}
    Assume $M \ra S$ is a Zariski locally trivial fibration with fiber $F$. Then $[M]=[S][F]$ in $K_0(Var/k)$. 
\end{lemma}
    
\begin{lemma} \label{lemma2}
   In the Grothendieck ring of varieties $K_0(Var/k)$, we have: 
   \begin{enumerate}
       \item $[\emptyset] = 0$
       \item $[Spec \ k] = 1$
       \item $[\P^n]= 1 + \L + \ldots + \L^n$, where $\L= [\mathbb{A}^1]$.
   \end{enumerate}
\end{lemma}
The above lemma is very easy to prove, see for example \cite{Kap}.

Poonen has shown that the ring $K_0(Var/k)$ is not a domain \cite{Ponnen}. Let us denote by $\mathbb{L}$ the equivalence class of the affine line. It has been further shown that $\mathbb{L}$ is a zero-divisor in the ring $K_0(Var/k)$ \cite{Lzero}, \cite{Lzero2}, \cite{Lzero3}. 

\begin{definition}
 Let $X$ and $Y$ be smooth projective varieties. They are called \textbf{$D$-equivalent} if their derived categories $D^b(X)$ and $D^b(Y)$ of bounded complexes of coherent sheaves are equivalent as triangulated categories, i.e., there exists an equivalence of categories $\phi : D^b(X) \ra D^b(Y)$ which commutes with the translations and sends any distinguished triangle to a distinguished triangle. 
\end{definition}

\begin{definition}
    Two algebraic varieties $X$ and $Y$ are called \textbf{$L$-equivalent} if the differences $[X]-[Y]$ vanishes in the localization $K_0(Var_{\C})[\L^{-1}]$. Equivalently, for some $n >0$ we have $([X]-[Y]).\L^n =0$. 
\end{definition}

\textbf{Question(Kuznetsov-Shinder)}: Assume X and Y are $D$-equivalent smooth projective varieties. What is the relation between the classes $[X]$  and $[Y]$ in the Grothendieck ring $K_0(Var/k)$? 

A partial answer to the above question of theirs was conjectured by them. 

\textbf{Kuznetsov-Shinder Conjecture:} If $X$ and $Y$ are smooth projective simply connected varieties such that $D(X)$ and $D(Y)$, then there is a non-negative integer $r \geq 0$ such that $([X]-[Y])\mathbb{L}^r = 0$. Equivalently, $D$-equivalence of simply connected varieties implies their $L$-equivalence. 

\begin{remark}
   The condition of being simply connected is necessary. Efimov \cite{Efimov} and independently Ito et. al \cite{Ito}, showed that there exist derived equivalent Abelian varieties over $\C$ which are not $L$-equivalent.  
\end{remark}

For a list of known positive cases of Kusnetsov-Shinder conjecture, we refer the reader to the survey by Ito et. al \cite{Ito}.

\subsection{An example in support: Rational 3-folds}
Lesieutre \cite{Lesi} gave an example of a rational threefold to show that a smooth projective variety can have infinitely many Fourier-Mukai partners, i.e., there are infinitely many non-isomorphic smooth projective varieties (all rational threefolds) whose derived category of coherent sheaves are all isomorphic. 

\begin{theorem}\cite{Lesi} Let $\textbf{p}$ denote an ordered $8$-tuple of distinct points in $\P^3$, and let $X_{\textbf{p}}$ be the blow-up of $\P^3$ at the points of $\textbf{p}$. There is an infinite set $W$ of configurations of $8$ points in $\P^3$ such that if $\textbf{p}$ and $\textbf{q}$ are distinct elements of $W$, then $D(X_{\textbf{p}}) \cong D(X_{\textbf{q}})$ but  are not isomorphic. 
\end{theorem}

A rational variety is simply connected (with the etale fundamental group being trivial or if defined over $\C$ then with the usual topological fundamental group) \cite[Corollary 4.18]{debarre}. In particular this example satisfies the assumptions of Kuznetsov-Shinder conjecture. We can easily show that this example provide support for the conjecture. The conjecture holds for it. 

\begin{proposition}
    Each pair of non-isomorphic derived equivalent $X_{\textbf{p}}$ and $X_{\textbf{q}}$ have the same class in the Grothendieck ring of varieties. In particular, they are $\L$-equivalent.  
\end{proposition}

\begin{proof}
    From \cite[Lemma 1]{Lesi}, we know that $X_{\textbf{p}}$ and $X_{\textbf{q}}$ are isomorphic iff $\textbf{p}$ and $\textbf{q}$ coincide upto an automorphism of $\P^3$ and permutation of points. From the alternate description of $K_0(Var/k)$ using blowups (see \cite{Bittner} and \cite{Looi}), we can write 
    \begin{equation}
        [X_{\textbf{p}}] = [\P^3] -8[k] + [E],
    \end{equation}
    where $E$ is the exceptional divisor and $[k]$ is the equivalence class of a point, which is our unit element. The exceptional divisor of the blow up is Zariski locally isomorphic to $\{pt\} \times \P^2$ (\cite[Theorem II 8.24]{Hartshorne}). Thus, from the relation above and lemma \ref{lemma2}, we get
    \begin{eqnarray*}
        [X_{\textbf{p}}]  &= \L^3 + \L^2 + \L + 1 -8 + \L^2+ \L + 1 \\
                        & = \L^3 + 2\L^2 + 2\L - 6.
    \end{eqnarray*}
    Since this formula, does not depend on the configuration of the 8 points in $\P^3$, we see that all the varieties $X_{\textbf{p}}$ have the same class in the Grothendieck ring of varieties and hence are $\L$-equivalent. 
\end{proof}

\begin{remark}
    Note that the above example shows that there can be infinitely many isomorphism classes of varieties which have the same equivalence class in Grothendieck ring of varieties.
\end{remark}

\subsection{Grothendieck Ring of Varieties and Birationality}

One can study rationality problems in algebraic geometry using the Grothendieck ring of varieties. Consider the following lemma:
\begin{lemma}\cite[Lemma 2.1]{SG}
Let $X, X'$ be smooth birationally equivalent varieties of dimension $d$. Then we have an equality in the Grothendieck ring $K_0(Var/k)$: 
\begin{equation*}
    [X'] -[X] = \L .\mathcal{M}
\end{equation*}
where $\mathcal{M}$ is a linear combination of classes of smooth projective varieties of
dimension $d-2$.    
\end{lemma}
This immediately gives the following rationality condition as a corollary:
\begin{corollary}
    If $X$ is a rational smooth $d$-dimensional variety, then
\begin{equation}
    [X] = [\P^d] + \L.\mathcal{M}_X
\end{equation}
where $\mathcal{M}_X$ is a linear combination of classes of smooth projective varieties of dimension $d - 2$.
\end{corollary}

\begin{remark}
    Localizing with respect to a zero divisor: 
\end{remark}

Further, there is another open conjecture of Larsen and Lunts \cite{LL}.\\

\textbf{Conjecture}: If $[X]=[Y]$, then $X$ is birational to $Y$.\\

The above conjecture is trivially true for the example of Lesieutre as we have constructed varieties which are birational to each other. The conjecture has been proven true for Calabi-Yau's by Liu and Sebag in \cite[Corollary 1]{LS}. A weaker result in this direction was already proven  by Larsen and Lunts in \cite{LL}:

\begin{theorem}(\cite{LL})
    The quotient-ring $K_0(Var/k)/\L$ is naturally isomorphic to the free abelian group generated by classes of stable birational equivalence of smooth projective connected varieties together with its natural ring structure. In particular, if $X$ and $Y_1, \ldots, Y_m$ are smooth projective connected varieties and 
    \begin{equation*}
        [X] = \sum_{j=1}^{m}n_j[Y_j] \mod \L
    \end{equation*}
for some $n_j \in \Z$, then $X$ is stably birationally equivalent to one of the $Y_j$.
\end{theorem}

Recall that two smooth projective varieties $X$ and $Y$ are called \textbf{stably birationally equivalent} if for some $m, n \geq 1$,  $X \times \P^m$ is birationally equivalent to $Y \times \P^n$.

\begin{remark}
    Generally the notion of stable birational equivalence is weaker than birational equivalence. However, for varieties with non-negative Kodaira dimension these two notions are the same \cite[Lemma 2.6]{SG}.
\end{remark} 

\subsection{Comparison with K-equivalence}

Kawamata introduced K-equivalence in \cite{Kawa1}.

\begin{definition}
    Let $X$ and $Y$ be smooth projective varieties. They are called \textbf{$K$-equivalent} if they are birationally equivalent and if there exists a smooth projective variety $Z$ with birational morphisms $f : Z \ra X$ and $g: Z \ra Y$ such that the pull-backs of the canonical divisors are linearly equivalent: $f^*K_X \sim g^*K_Y$ .
\end{definition}

He further conjectured the following \\

\textbf{Kawamata's DK-Conjecture:} \cite{Kawa1} Let X and Y be birationally equivalent smooth projective
varieties. Then the following are equivalent.
\begin{enumerate}
    \item There exists an equivalence of triangulated categories $D^b(X) \cong D^b(Y)$.
    \item There exists a smooth projective variety Z and birational morphisms $f : Z \ra X$ and $g : Z \ra Y$ such that $f^*K_X \sim g^*K_Y$.
\end{enumerate}
In other words,for birational varieties D-equivalence is the same as K-equivalence. For a survey of the current known examples for this conjecture along with its more general form, see \cite{Kawa2}.

The above example of Lesieture satisfies the DK-conjecture as well. This can be seen easily by taking $Z$ to be the blowup of $\P^3$ at all the distinct points in the sets $\textbf{p}$ and $\textbf{q}$ of $X_{\textbf{p}}$ and $X_{\textbf{q}}$. The maps $f$ and $q$ are just the respective blow downs. Note that pullback of the canonical divisors of $X_{\textbf{p}}$ and $X_{\textbf{q}}$ are the same. 

\section*{Acknowledgements}
This work was prepared for talk at 38th Annual Conference of
Ramanujan Mathematical Society, 2023 and is supported by the SERB POWER grant No. SPG/2021/002307.

\end{document}